\documentclass[11pt]{article}

\usepackage[utf8]{inputenc}
\usepackage{amsmath}
\usepackage{amsthm}

\usepackage{graphicx}
\usepackage{amssymb}
\usepackage{hyperref}

\numberwithin{equation}{section}
\newtheorem{thm}{Theorem}
\newtheorem{assumption}{Assumption}
\newtheorem{lemma}{Lemma}

\newcommand{\tr}{\mathrm{tr}}
\usepackage[mathscr]{euscript}
\DeclareSymbolFont{rsfs}{U}{rsfs}{m}{n}
\DeclareSymbolFontAlphabet{\mathscrsfs}{rsfs}

\DeclareMathOperator*{\argmin}{arg\,min}

\date{\today}
\usepackage{natbib}
\usepackage{graphicx}

\begin{document}

\title{Shuffled total least squares}

\author{Qian Wang and
    Daniel Sussman\footnote{Email: sussman@bu.edu}\\
    Boston University, Department of Mathematics and Statistics}

\maketitle

\begin{abstract}
Linear regression with shuffled labels and with a noisy latent design matrix arises in many correspondence recovery problems.
We propose a total least-squares approach to the problem of estimating the underlying true permutation, and provide an upper bound to the normalized Procrustes quadratic loss of the estimator.
We also provide an iterative algorithm to approximate the estimator and demonstrate its performance on simulated data.
\end{abstract}

\section{Introduction} \label{s:intro}

Correspondence and matching problems arise in numerous fields including computer vision \citep{CD16,Szeliski,Hartley},  network alignment \citep{Chen,Elmsallati,Klau,Zaslavskiy}, and sequence dating \citep{FMR16,Pet99,Ken63,Ken69,Ken70}.
Recently assignment problems have been receiving increased attention from statisticians, applied mathematicians, and computer scientists due to the increased availability of data from all sources.
Merging data sets enables researchers to answer questions that could not be answered with separate ones.
These problems typically involve the observation of two sets of variables where the correspondence between observations in one set to the other is unknown or obscured.
A prototypical example of this is the observation of two sets of point clouds associated with the same entities and the goal is to align the point clouds and discover the correspondence between the entities.

The classical problem of linear regression address the first aspect, aligning the two sets of point.
Given the input feature matrix $X\in \mathbb{R}^{n\times p}$ and the output matrix $Y\in \mathbb{R}^{n \times q}$, the statistical model then takes on the form $$Y = X\beta + E$$ where $\beta\in\mathbb{R}^{p\times q}$ are the unknown coefficients, $E\in \mathbb{R}^{n\times q}$ is the noise usually assumed to be Gaussian.

A variant of the standard set up is the case when the true correspondence between the outputs and the inputs is unknown. That is, we observe instead a permuted data set $\{x_i,y_{\pi(i)}\}_{i=1}^n$, where $\pi$ is an unknown permutation. This is known as shuffled linear regression \citep{PWC16,FMR16,PWC17,CD16}, in which case the model being considered takes on the form
$$Y = \Pi X^T\beta + E$$
where $\Pi$ is an unknown permutation matrix.
Shuffled linear regression is the natural framework for analyzing experiments that
simultaneously involves a large number of objects such as flow cytometry \citep{APZ17}; it has also been used in situations when the order in which the measurements are taken is uncertain such as archaeological measurements \citep{R51}.

Another variant of the classical linear regression is known as the total least-squares regression \citep{GL80} where instead of observing the true design matrix $X$, we observe a noisy version of it, $Y_1$. The model hence takes on the form 
\begin{align*}
Y_1 &= X + E_1 \\
Y_2 &= X\beta + E_2.
\end{align*}
Thus making $X$ a latent design matrix that is no longer known. 
This assumption is frequently more realistic in real-world applications due to samplings errors, human errors, modeling errors, etc, and the resulting model is considered a type of ``errors-in-variable'' model in statistics.

While the regression problem with errors in the predictor variables is similar to the regular regression problem, the standard method of using ordinary least-squares is no longer optimal for the errors in variables case.
To see this, suppose we estimate the coefficients using ordinary least squares.
Then we will have $\hat{\beta}_{ols}=(Y_1^TY_1)^{-1}Y_1^TY_2$, the expectation of which is $\mathbb{E}(\hat{\beta}_{ols})=\beta-(n-p-1)(X^TX)^{-1}S\beta$, where $p$ is the dimension of $X$ and $S=\frac{1}{n}\mathbb{E}(E_1^TE_1)$ \citep{HM72}.
Hence, the OLS estimator is biased and indeed inconsistent, so instead we use ``total least-squares'' which takes into consideration noise from both the input and the output in its estimation methods. The detail are discussed in section~\ref{Section: 2.1}.

In this paper, the model under consideration combines these two variants. That is, we consider the model with both shuffled labels and noisy latent design matrix. Formally, let $Y_1,Y_2\in \mathbb{R}^{n\times p}$ represent the observations, and $R\in \mathbb{R}^{p\times p}$ represent the linear transformation matrix, then the model takes on the form
\begin{align}
\label{eq:tls-model}
    Y_1 &= X + E_1\\
    Y_2 &= \Pi^*XR + E_2,\notag
\end{align}
where $\Pi^*$ is an unknown permutation matrix, $X\in\mathbb{R}^{n\times p}$ is a latent design matrix, and $E_1,E_2\in \mathbb{R}^{n\times p}$ are the noise matrices.

Model \eqref{eq:tls-model} has applications in many real-world situations and here we discuss two broad areas.
The first line of application lies in the same realm as  applications discussed in \citet{PWC16}, which studies shuffled linear regression where $X$ is observed without noise.
In the pose and correspondence estimation problem, images of the same object can be taken during different conditions such as different angles and positions, and in order to align them, keypoints such as corners, edges, or other features, which are invariant to those changes, are detected by considering an area of certain pixel intensities around it \citep{David2004-dd}.
After detecting the keypoints in the two images individually, one can recover the linear transformation between two given images of a similar object through matching the detected keypoints.

In practice, image recording devices fail to record the intensity of a given image scene exactly, resulting in random noise or blurring in the recorded scenes~\citep{HP06}.
Therefore, the shuffled total least-squares regression setting is a more suitable model as the keypoints in both images will be estimated with noise.
More generally, many alignment problems involve aligning objects which are all observed with noise where an errors-in-variables approach would be a more accurate problem representation.

A second area of application pertains to the alignment of two embeddings such as graph embeddings, word embedding, sequence embeddings, and so on, \citep{CX19,CT17,CM15,BB19}.
For example, the graph matching problem considers finding the correspondence between nodes across two or multiple graphs and one way of doing so is through the matching of the graph embeddings.
As a specific instance, suppose one observes an adjacency matrix $A\in\{0,1\}^{n\times n}$ representing a friendship network from Facebook where the nodes represent users and the edges represent a connection, and a bipartite adjacency matrix $B \in \{0,1\}^{n \times m}$ representing a purchasing network from Amazon, where the nodes represent $n$ shoppers and $m$ products and the edges represent a purchase.
One is tasked to find the user correspondence across these two networks.
A common method to do so is to match the node embeddings of the two adjacency matrices \citep{Heimann2018-ci,Liu2020-av,Chu2019-yv,Sun2020-sm,Nelson2019-zc}, see also \citep{Sussman14,Athreya16}.
Suppose each user has a feature representation which can be estimated from the graph, then ideally the estimated feature matrix from one graph would be a transformed version of another.
With noise present and the labels shuffled, model \eqref{eq:tls-model} can be used as a framework to recover the correspondence.

As another example, In computational biology one challenge is to align the nodes between two protein-protein interaction (PPI) networks where the nodes represent proteins and edges represent the existence of an interaction between two proteins \citep{Fan19}.
The structure of these networks is becoming increasingly well-known, especially with the high-throughput interaction detecting devices.
However, as the reads are often subjected to error, the mis-identification of interactions between proteins generates graphs that are inevitably noisy representations of the true underlying protein structures.
If using an embedding approach for matching, we are again in the settings of model \eqref{eq:tls-model}.

Note, there are some cases when an OLS approach is more appropriate, an example is the sensor network setting discussed by \citet{PWC16}.
In large sensor networks, it is often the case that the bandwidth sent from the sensors to the fusion center is dominated by the bandwidth that serves to identify the sensor rather than the information being transferred.
A header-free communication scheme is proposed to alleviate this issue \citep{KSFAD09}, under which one is tasked with the recovery of the sensor identities based primarily on the sensor observations.
A shuffled OLS approach is appropriate if the sensors are placed in fixed locations that can be highly accurately measured at the time of installation.
Hence, there is essentially no noise in the design matrix, making an OLS suitable.

\subsection{Contributions} \label{ss:contributions}

This paper addresses the correspondence estimation problem in the linear regression setting with both shuffled output labels and corrupted input data matrix as in model \eqref{eq:tls-model}.
We focus on the linear model \eqref{eq:tls-model} with a latent but fixed design matrix $X\in \mathbb{R}^{n\times p}$, a unknown coefficient matrix $R\in \mathbb{R}^{p\times p}$, and Gaussian noises $E_1,E_2\sim N(0,\Sigma_p)$.
We estimate $\Pi$ via the total least-squares method which gives $\hat{\Pi}=\argmin_{\Pi\in\mathcal{P}_n}\sum_{i=p+1}^{2p}\sigma_i([Y_2|\Pi Y_1])$, and we evaluate $\hat{\Pi}$ by the normalized Procrustes quadratic loss defined as $\frac{1}{\|X\|_F^2}\min_{Q\in \mathcal{O}(p)} \| \Pi^*X-\hat{\Pi} X Q\|_F^2$.
We provide an upper bound of this loss as a function of the latent design matrix $X$, the coefficient matrix $R$, the dimensions $(n,p)$, and the covariance matrix $\Sigma$.
We compare the result with previous literature and show that our bound compares favorably with previous bounds shown for related problems.
Computationally, since the above estimate is NP-hard to compute, we compare, via simulations, the performance of four proposed approximation methods.
We use an oracle-type estimator to examine the upper bound provided by the main result. 

\subsection{Organization} \label{ss:org}

The remainder of the paper is organized as follows.
In the next section, we present the problem background and related work.
In Section~\ref{s:problem}, we define notation and formally state the problem of interest.
The main result and the discussion of implications are presented in Section~\ref{s:main}.
Section~\ref{s:sim} presents the iterative methods for approximating the estimates and empirical results from simulations in various contexts.
Appendix~\ref{Section: 5} has proofs the proofs of the main theorem; the supporting lemmas are stated and proved in Appendix~\ref{app:tech_lemma}.

\section{Background and Relative Work} \label{s:background}
This section reviews work related to the total least-squares method and the problem of latent permutation estimation. 
We first briefly overview notations used throughout the paper.
We use $\Pi$ to denote permutation matrices, and  $\mathcal{P}_n$ to denote the set of $n\times n$ permutation matrices.
$Y_\Pi$ represents the matrix $Y$ with rows permutated according to the permutation $\Pi$. The $k$th column of the matrix $A$ is denoted by $a_k$, and the entries are denote $a_{jk}$ or $A_{jk}$. $\mathcal{O}(p)$ denotes the set of orthogonal matrices of dimension $p$.
The eigenvalues of a square matrix $A$ are arranged in weakly descreasing order: $\lambda_{max}(A)=\lambda_1(A)\ge \lambda_2(A)\ge \dotsb \ge \lambda_n(A)=\lambda_{min}(A)$.
Likewise, the singular values of a matrix $B$ with rank $r$ are ordered $s_1(B)\ge s_2(B)\ge \dotsb \ge S_r(B)$.
We use $\|B\|_*$ and $\|B\|_F$ to denote the trace and Frobenius norm of a matrix $B$, and $c,c_1,c_2$ to denote universal constants that may change from line to line.

\subsection{ Total least-squares and errors-in-variables}
\label{Section: 2.1}
To clarify the total least-squares (TLS) aspect of the problem, let us first assume that the permutation matrix $\Pi^*$ is known and equal to the identity.
TLS regression was first introduced by \citet{GL80} as a natural generalization of the ordinary least-squares (OLS) method when both the input data and output data are being perturbed by noise.
The model under consideration therefore becomes $Y_1 = X+E_1; Y_2=XR+E_2$.
In this setting, both $X \in \mathbb{R}^{n\times p}$ as well as $R\in \mathbb{R}^{p\times p}$ are unobserved.
Instead, we observe $(Y_1,Y_2)\in(\mathbb{R}^{n \times p}\times\mathbb{R}^{n\times p})$.

To estimate the unknown quantities, the OLS method seeks the solution to the optimization problem
\begin{align}
    &\min_{\hat{R}\in \mathbb{R}^{p \times p}}\|Y-X\hat{R}\|^2_F
\end{align}
The TLS method, as a natural generalization of the OLS method, does so in solving the following optimization problem: 
\begin{align}
\label{eq:tls-opt}
    &\min_{\hat{X}\in \mathbb{R}^{n \times p},\hat{R}\in\mathbb{R}^{p \times p}}\|[Y_2-\hat{X}\hat{R}|Y_1-\hat{X}]\|^2_F.
\end{align}
The denoised observations are $\hat{Y_1}=\hat{X}$ and $\hat{Y_2}=\hat{X}\hat{R}$.
As $\hat{Y_2}=\hat{Y_1}\hat{R}$, we have $\mathrm{rank}([\hat{Y_2}|\hat{Y_1}])\le p$, therefore the optimization problem \eqref{eq:tls-opt} is in fact equivalent to a matrix low-rank approximation problem  \citep{MH07},
\begin{align}
\label{eq:low-rank}
    &\min_{\hat{Y}_1, \hat{Y}_2 \in \mathbb{R}^{p\times p}}\|[Y_2|Y_1]-[\hat{Y_2}|\hat{Y_1}]\|^2_F\notag\\
    & s.t.~\mathrm{rank}([\hat{Y_2}|\hat{Y_1}])\le p.
\end{align}
Problem \eqref{eq:low-rank} can be understood as a search for the best approximation of the data $[Y_2|Y_1]$ by a matrix of a lower rank.

The solution is given by \citet{EY36} as follows.
Let $Y=[Y_2|Y_1]=U_Y\Sigma_Y V_Y^T$ be the reduced singular value decomposition of $[Y_2|Y_1]$, where $U_Y\in \mathbb{R}^{n \times 2p}$ and $V_Y\in \mathbb{R}^{2p \times 2p}$ have orthonormal columns and $\Sigma_Y\in \mathbb{R}^{2p \times 2p}$ is diagonal with entries $\sigma_1([Y_2|\Pi^*Y_1])\ge \sigma_2([Y_2|\Pi^*Y_1])\ge \dotsb
\ge \sigma_{2p}([Y_2|\Pi^*Y_1])\ge 0$, with $\sigma_i([Y_2|\Pi^*Y_1])$ being the $i$th singular value of $[Y_2|\Pi^*Y_1]$.
Let $\Tilde{U}_Y\in \mathbb{R}^{n \times p}$ and $\Tilde{V}_Y\in \mathbb{R}^{2p \times p}$ be, respectively, the first $p$ columns of $U_Y$ and $V_Y$, and $\Tilde{\Sigma}_Y\in \mathbb{R}^{p \times p}$ be the diagonal matrix with diagonals $\sigma_1([Y_2|Y_1]),\dotsc, \sigma_p([Y_2|Y_1])$. 
The minimizer of \eqref{eq:low-rank} is $[\hat{Y_2}|\hat{Y_1}]=\Tilde{U}_Y\Tilde{\Sigma}_Y\Tilde{V}_Y^T$, so $\hat{Y}_2=\Tilde{U}_Y\Tilde{\Sigma}_Y\Tilde{V}_Y^T[,1:p]$ and $\hat{Y}_1=\Tilde{U}_Y\Tilde{\Sigma}_Y\Tilde{V}_Y^T[,p+1:2p]$.

Hence, we have $\hat{X}=\hat{Y}_1$ and $\hat{R} = (\hat{Y}_1^T \hat{Y}_1)^{-1}\hat{Y}_1^T \hat{Y}_2$.
The optimal value of the minimization of \eqref{eq:tls-opt} and \eqref{eq:low-rank} is
\begin{align}
\label{eq:tls-loss}
    \sum_{i=p+1}^{2p}\sigma^2_i([Y_2|Y_1]).
\end{align}

The TLS estimator $(\hat{X},\hat{R})$ is known to be the maximum likelihood estimator in the ``errors-in-variable'' model under the assumption that $vec([Y_2|Y_1])$ is a zero-mean, normally distributed random vector with a covariance that is a multiple of the identity \citep{MH07}.

\subsection{Latent Permutation Estimation}\label{sec:latent_perm}

While there is little work in the linear regression setting with both corrupted inputs and outputs, there are many related works in the non-noisy input setting in which permutation recovery has been studied.
In \citet{CD16}, the problem of matching two sets of noisy vectors was studied statistically.
There the model has the form $Y_1=X + E$ and $Y_2 = \Pi^*X +E$, where $X\in \mathbb{R}^{n\times p}$ is an unknown matrix and $\Pi^*\in \mathcal{P}_n$ is an unknown permutation matrix.
The goal is to recover $\Pi^*$.
\citet{CD16} establishes minimax rates on the separation distance between rows of $X$ that allows for exact permutation recovery by studying the least-squares estimator defined by $\hat{\Pi}=\argmin_{\Pi\in\mathcal{P}_n}\|\Pi Y_2-Y_1\|_F^2$.

The problem of permutation estimation has been studied in more detail for the case that $X$ is observed without noise.
\citet{PWC16} considered the one-dimensional case, where the model is of the form $y=\Pi^*Xr^*+e$ where $r^*\in\mathbb{R}^p$ is an unknown vector, $X\in\mathbb{R}^{n\times p}$ is a design matrix of i.i.d.\ standard Gaussian variables, $\Pi^*$ is an unknown $n\times n$ permutation matrix, and $e\in \mathbb{R}^n$ is observation noise. 
The paper provides necessary and sufficient conditions, in terms of the signal-to-noise ratio, under which the probability of exactly recovering the true permutation goes to one.
Computationally, $\hat{\Pi}$ can be computed exactly in polynomial time according to a sorting algorithm. 
A multivariate version of this model, $Y=\Pi^*XR+E$, was considered in \citet{PWC17} where the interest is not in the exact permutation recovery, but in the prediction error defined by $\|\hat{\Pi}X\hat{R}-\Pi^*XR\|_F^2$.
The paper characterizes the minimax prediction error by analyzing the maximum likelihood estimator  $(\hat{\Pi},\hat{R})=\argmin_{\Pi\in \mathcal{P}_n;R\in \mathbb{R}^{d\times  m}}\|Y-\Pi XR\|_F^2$.

The statistical seriation problem considers permuting the rows of a matrix so that its columns have the same shape such as monotone increasing or unimodal.
In this setting, one observes only a single matrix $Y$, with $Y=\Pi^*X^*+E$ where the columns of $X$ are constrained to be unimodal.
\citet{FMR16} established minimax rates for the prediction error  $\|\hat{\Pi}\hat{X}-\Pi^*X\|_F^2$ by analyzing the least-squares estimators.

Another closely related line of research comes from \citet{Dai2019-on,Dai2020-dh} who analyze the problem of Gaussian database alignment.
The model that they study is very similar to ours however there are some key diffferences.
First, from our vantage point, they consider the case that the latent design is itself drawn from a Gaussian distribution whereas our results are for a fixed latent design matrix.
This change in setting enables the second major difference which is that \citeauthor{Dai2019-on} consider the problem of permutation recovery from the perspective of Hamming distance from the truth whereas our work analyzes the Procrustes quadratic loss.
While it may be possible to use the results on the Procrustes quadratic loss to derive bounds on the hamming distance for certain latent designs, we did not explore that in this work.
Finally, the estimation technique considered is a likelihood approach which is identical to our approach under certain assumptions on the covariance of the noise however will not in general correspond.
The authors consider a canonical setting under which an exact computation of the maximizer is possible in polynomial time.

\section{Problem Setting and Evaluation Criteria}\label{s:problem}

In this section, we discuss in detail about the key assumptions of the model, the total least-squares estimator, and the metric to evaluate its performance.

\subsection[Problem Setting and Total Least-Squares Estimator]{Problem Setting \& Total Least-Squares Estimator}

In the introduction section, we have motivated model \eqref{eq:tls-model} as a type of errors-in-variable model. We now state our key assumption on each of the components in the model.

\begin{assumption}
\label{Assumption:1}
(Latent Design Matrix) We assume that the latent design matrix $X\in \mathbb{R}^{n\times p}$ has condition number 1, that is, $\kappa(X)=\frac{\sigma_1(X)}{\sigma_p(X)}=1$.
\end{assumption}

\begin{assumption}
\label{Assumption:2}
(Coefficient Matrix) We assume $\sigma_p(R)\le 1$ and $\sigma_1(R)\ge 1$. 
\end{assumption}

\begin{assumption}
\label{Assumption:3}
(Noise Variables) The noise variables $E_1$, $E_2$ are independent with i.i.d. rows distributed as $\mathcal{N}(\Sigma)$ for some $p\times p$ covariance matrix $\Sigma$.
\end{assumption}

Assumption~\ref{Assumption:1} is strong but is mostly needed to ensure a relationship between the Procrustes quadratic loss and the TLS error as we describe below.
Note, if the rows of $X$ are i.i.d Gaussian with mean zero and covariance matrix $\mathbf{C}$, then the condition number constraint will hold approximately if $\kappa(\mathbf{C})=1$ \citep{GT14}. In the case when $\kappa(X)\ne 1$, we could work on a transformed version of $Y_1$. Let $Y_1 = U_{Y_1}S_{Y_1}V_{Y_1}^T$, and define the transformed obervations as $Y_1^{new}=U_{Y_1}$. This way we have $\kappa(Y_1^{new})=\kappa(U_{Y_1})=1$.

In many cases we might not have the condition number condition met.
In such cases when $\kappa(X)\ne 1$, we could instead work on a transformed version of $Y_1$. Let $$Y_1 = U_{Y_1}S_{Y_1}V_{Y_1}^T,$$ and define the transformed observations as $$Y_1^{new}=U_{Y_1}=Y_1V_{Y_1}S_{Y_1}^{-1}.$$
This way we have $\kappa(Y_1^{new})=\kappa(U_{Y_1})=1$ which will generally lead to $\kappa(XV_{Y_1}S_{Y_1}^{-1}) \approx 1$.
Alternatively, the lost result could be stated in terms of similarly normalized version of $X$, namely $XV_XS_X^{-1}$.

Assumption \ref{Assumption:2} is an assumption on the singular values of the unknown linear transformation matrix $R$. This assumption is, as will be clear in the proof section, for mathematical convenience only. We assume this so that we can write in short $\max\{1,\sigma_1(R)\}$ as $\sigma_1(R)$, and $\min\{1, \sigma_p(R)\}$ as $\sigma_p(R)$. The main result can be extended to $R$ with arbitrary singular values by replacing the original forms back.

In Assumption \ref{Assumption:3} we assumed the normality of the noise variables. 
Recall that the total least-squares estimator is the maximum likelihood estimator in the errors-in-variables model if the noise term has a mean zero Gaussian distribution with a covariance that is of the multiplication of the identity \citep{MH07}.
However, since this situation is rarely met in real applications, we instead consider a more general case when $E\sim N(0,\Sigma_p)$, that is, we allow possible correlations between variables. While the total least-squares approach may not be the most optimal approach in this setting, it is useful to understand its performance and this will inform future approaches that take into account this asymmetric variance.

\subsection{Evaluation Metric}In this section we first introduce the total least-squares estimator of the latent permutation matrix, and then we state the proposed criteria for estimator evaluation, namely, we evaluate its performance by a quantity named normalized Procrustes quadratic loss, and finally we explain the rational behind this choice.

Let $Y_\Pi=[Y_2|\Pi Y_1], M_\Pi = [\Pi^*XR|\Pi X]$, and $E_\Pi=[E_2|\Pi E_1]$, we could writing model \eqref{eq:tls-model} equivalently as
\begin{align}
\label{eq:tls-model2}
    Y_\Pi &= M_\Pi + E_\Pi.
\end{align}
To estimate $\Pi^*$ using the total least-squares method, recall that in section \ref{Section: 2.1} we have phrased the optimization problem in total least-squares regression as a matrix low-rank approximation problem, and the optimal value of the minimization of loss is the sum of the least $p$ squared singular values of the concatenated matrix of the two observation matrices, namely, $Y_2$ and $Y_1$. In a similar fashion, now with the permutation matrix $\Pi$ introduced, the total least-squares method seeks the optimal permutation matrix which minimizes the loss term.  Hence, the estimator is given by
\begin{align}
\label{eq:tls-estimator}
    \hat{\Pi}=\argmin_{\Pi\in \mathcal{P}_n}\sum_{i=p+1}^{2p}\sigma_i^2(Y_\Pi).
\end{align}
Though our main interest lies in estimating the underlying permutation matrix, it is worth mentioning that once a $\hat{\Pi}$ is found, the model can be viewed as a regular errors-in-variables model, and one can estimate the latent design matrix $X$ and the unknown linear transformation matrix $R$ using the total least-squares method.

In Section \ref{s:background}, we have seen multiple criteria being proposed to evaluate an estimator $\hat{\Pi}$. Two of the most commonly ones used are the hamming distance between $\hat{\Pi}$ and $\Pi^*$ defined as $d_H(\hat{\Pi},\Pi^*)=\#\{i|\hat{\Pi}(i)\ne \Pi^*(i)\}$,
and the normalized quadratic loss defined as 
$\frac{1}{np}\|\hat{\Pi}X-\Pi^*X\|_F^2$.
The Hamming distance is intuitively appealing but the structure of $X$ can lead to problems with this metric.
Specifically, if $X$ has multiple identical rows, $\Pi^*$ becomes non-identifiable and there are multiple $\hat{\Pi}$ which would make the quadratic loss zero.

While in general a small quadratic loss does not guarantee a small Hamming distance, provided the rows of $X$ are sufficiently separated, they will be strongly correlated. 
In the settings of total least-squares, a natural generalization of the normalized quadratic loss is the normalized Proscrustes quadratic loss, defined as
$$\frac{1}{\|X\|_F^2}\min_{Q\in \mathcal{O}(p)} \| \Pi^*X-\hat{\Pi} X Q\|_F^2.$$
That is, in the evaluation of the TLS estimator, we allow for an orthogonal rotation of the predicted value.

Allowing for orthogonal rotation, that is, evaluating estimators in after a Procrustes alignment, is a common practice in total least-squares estimation method.
Depending on the problem at hand, there might be a variety of reasons to do so.
Here the allowance of this orthogonal rotation is due to the fact that the total least-squares estimator $\hat{\Pi}$ might not lead to a small quadratic loss,  $\|\hat{\Pi}X-\Pi^*X\|_F^2$,  due to lack of identifiability.

To see this, let us consider a noiseless case when $E_1=E_2=0$, and suppose we have a design matrix with the following structure
\begin{align*}
   X = \begin{bmatrix}
    \mathbf{1}_5&-\mathbf{1}_5\\
    \mathbf{1}_5&\mathbf{1}_5
\end{bmatrix},
\end{align*}
where $\mathbf{1}_5$ is a column vector of $1'$s of size 5.
For convenience we assume both $\Pi^*$ and $R$ are identity matrices.
Consider the permutation matrix 
\begin{align*}
    &\Tilde{\Pi} = \begin{bmatrix}
   0&I_{5}\\
   I_{5}&0
   \end{bmatrix}
\end{align*}
Note that since $\tilde{\Pi} X = X\begin{bmatrix}
    1 & 0 \\ 0 & -1
\end{bmatrix}$, it holds that
\begin{align*}
    \sum_{i=p+1}^{2p}\sigma_i([Y_2|\Tilde{\Pi} Y_1]) &=\sum_{i=p+1}^{2p}\sigma_i([X|\Tilde{\Pi} X]) \\
    & = \sum_{i=p+1}^{2p}\sigma_i([X|{\Pi}^* X])= \sum_{i=p+1}^{2p}\sigma_i([Y_2|{\Pi}^* Y_1])=0.
\end{align*}
Here the Hamming distance between $\tilde{\Pi}$ and $\Pi^*$ is large and the quadratic loss is also large, $\frac{1}{np}\|\Pi^*X-\Tilde{\Pi}X\|_f^2=2$, but both are minima for the TLS problem.

For the Procrustes quadratic loss, 
 $$\frac{1}{\|X\|_F^2}\min_{Q\in \mathcal{O}(p)} \| X-\hat{\Pi} X Q\|_F^2 =\frac{1}{np} \|X-\hat{\Pi} X Q^\dagger\|_F^2=0,$$
where $Q^\dagger=\Tilde{Q} = \begin{bmatrix}
    1&0\\
    0&-1
    \end{bmatrix}.$\\
However, when $\kappa(X)=1$, Lemma~\ref{lem:pql_svals} guarantees the relationship
$$\min_{Q\in \mathcal{O}(p)} \| X-\Pi X Q\|_F^2 \le2\sum_{i=1+p}^{2p} \sigma_i^2(X|\Pi X).$$
This relationship tells us that the total least-squares method, which seeks to minimize the sum of the least $p$ squared singular values, produces an estimate such that the Procrustes quadratic loss is at most half the optimal objective value of the total least-squares problem.
Therefore, without putting further shape or distributional constraints on $X$, in our settings, it is more suitable to choose the Procrustes quadratic loss as the loss measure.

\section{Main Results}\label{s:main}
In this section, we state our main results and discuss some of their consequences.

Proof of the theorem can be found in Section \ref{Section: 5}.
In this section we state and prove our main result which provides an upper bound on the normalized Procrustes quadratic loss of the total least squares estimator. 
\subsection{Main Result}
Our main theorem provides an upper bound on the loss measure in terms of $\Sigma$, $X$, $R$, and the dimensions $(n,p)$, with $c$ denoting an absolute constant.
\begin{thm}
\label{Theorem: 1}
For the statistical model \eqref{eq:tls-model}, under the assumptions \ref{Assumption:1} - \ref{Assumption:3} , the total least squares estimator $\hat{\Pi}$ as stated in \eqref{eq:tls-estimator} satisfies
\begin{align}
\label{eq:main-result}
     &\frac{\min_{Q\in \mathcal{O}(p)} \| \Pi^*X-\hat{\Pi} X Q\|_F^2}{\|X\|_F^2}\notag\\
    &\le  \frac{2p}{\sigma_p^2(R)\|X\|_F^2}\left(1+\eta a_n\right)\lambda_1(\Sigma)\left[16\sigma_1(R) \|X\|_F\sqrt{2n} +2n\right],
\end{align}
where 
$a_n = \sqrt{\frac{\tr(\Sigma)}{\lambda_1(\Sigma)}\frac{\log(n)}{c n}}$,
with probability greater than  
\begin{align*}
1 - n^{-\eta^2},
\end{align*}
where $c$ is at least $\frac{1}{32}$.
\end{thm}
In the following paragraphs, we first discus where the two terms in the upper bound comes from. Then to aid in interpretation of the upper bound, we view it in terms of signal-to-noise ratio. And finally, we compare the upper bound to similar results in the shuffled linear regression setting.

We begin by loosely interpreting the sources of the two components in the upper bound without delving in to the specificity. The details on how each component is derived will be made clear in the proof section.

The first term in the upper bound,
$$\frac{2p}{\sigma_p^2(R)\|X\|_F^2}\lambda_1(\Sigma)16\sigma_1(R) \|X\|_F\sqrt{2n}, $$
comes from our bound on the interaction terms of $\hat{\Pi}X$ with the noise variables $E_1$, $E_2$.
The scale of this term depends on the norm of the latent design matrix $X$ relative to that of the sample size $n$, and the noise level indicated by $\lambda_1(\Sigma)$. 
If we assume that the noise level remains fixed and that the entries of $X$ are i.i.d. standard Gaussian, then we have  $\|X\|_F^2 = \Theta(np)$,%
for fixed $p$.
Then the first term is constant and approximately $32\sqrt{2p}\lambda_1(\Sigma)\frac{\sigma_1(R)}{\sigma_p^2(R)}$.

The second component in the upper bound then comes from our bound on the noise variables $E_1$ and $E_2$, specifically a bound on the largest eigenvalue of $E_{\Pi^*}^TE_{\Pi^*}$, i.e., $\lambda_1(E_{\Pi^*}^TE_{\Pi^*})$, where $E_{\Pi^*}=\begin{bmatrix}
    E_2|\Pi^*E_1
    \end{bmatrix}.$
When $\|X\|_F^2 = \Theta(np)$, the second term becomes $$\frac{4a_n\eta\lambda_1(\Sigma)}{\sigma^2_p(R)}$$
and goes to zero like $\sqrt{\frac{\log n}{n}}$ as the sample size goes to infinity. 

To further interpretation, in the context of our model, we define the signal-to-noise ratio to be $\mathrm{snr}=\frac{\|X\|_F^2}{n\tr(\Sigma)}$.
To simplify the notation, we replace $\tr(\Sigma)$ with the upper bound $p\lambda_1(\Sigma)$.

Then $a_n$ can be written as
$a_n = \sqrt{\frac{p\log(n)}{c n}}$,
and accordingly, the first term in the upper bound can be written as 
$c_1(R)\sqrt{\frac{\tr(\Sigma)}{\text{snr}}}$
where $c_1(R)=\frac{32\sqrt{2}\sigma_1(R)}{\sigma_p^2(R)}$.
Similarly, the second term in the upper bound can be written as $c_2(R)\sqrt{\frac{p\log(n)}{ n}}\frac{1}{\text{snr}}$,
where $c_2(R)=\frac{4\eta}{\sigma_p^2(R)}$.
Putting the two terms together, the upper bound can be viewed in terms of signal-to-noise ratio as
\begin{align*}
c_1(R)\sqrt{\frac{\tr(\Sigma)}{\text{snr}}}+c_2(R)\sqrt{\frac{p\log(n)}{ n}}\frac{1}{\text{snr}}.
\end{align*}

As before, the first term is the dominant term for fixed $p$. Therefore, in order to have the normalized Procrustes quadratic loss of the estimator go to zero, it is required that the signal-to-noise ratio goes to infinity. 

Finally, recall that assumption \ref{Assumption:2}, $\sigma_p(R)\le 1\ge \sigma_1(R)$, is for notational parsimony only
The main result \eqref{eq:main-result} can be extended to $R$ with arbitrary singular values by replacing $\sigma_p(R)$ with $\min\{1,\sigma_p(R)\}$, and $\sigma_1(R)$ with $\max\{1,\sigma_1(R)\}$. The details are stated in the proof section.

To view our upper bound in a broader light, we compare our result to some of the existing results in the shuffled linear regression setting, i.e., when $X$ is observed without noise.
As introduced in section \ref{s:background}, \citep{PWC16} considered the exact permutation recovery problem in the one-dimensional shuffled linear regression setting. They established that if the signal-to-noise ratio scales as $n$ then exact permutation recovery is achieved with high probability, that is, the probability of the event $\{\hat{\Pi}_{OLS}\ne \Pi^*\}$ goes to zero as the sample size goes to infinity. 
Though the model setting for shuffled linear regression is in some ways simpler than our model setting, the goal of exact permutation recovery is more difficult and hence requires a much more strict signal-to-noise ratio condition.

Continue their work in the shuffled linear regression, the same group of authors in \citep{PWC17} considered the permutation recovery problem in the higher dimensional setting where they do not evaluate the estimator in terms of exact recovery, but instead consider a quantity similar to ours which is the normalized quadratic loss written as $\frac{1}{np}\|\hat{\Pi}X\hat{R}-\Pi^*XR^*\|_F^2$.
Given that $p\le \log n$, their bound has a leading term which is the noise variance $\sigma^2$, which agrees with our upper bound.

Another closely related problem is the statistical seriation context as presented in section \ref{s:intro}.
Stated in our notation, \citet{FMR16}  derived an upper bound on $\frac{\|\Pi^*XR-\hat{\Pi}\hat{X}\hat{R}\|_F^2}{np}$  where $(\hat{\Pi},\hat{X}\hat{R})$ are derived through the ordinary least-squares method. They established that, for fixed $p$, the normalized quadratic loss has an upper bound with a leading term $\log(n)$. This might seem a weaker result compared with that of \citep{PWC17} as well as ours, but they are estimating $XR$ having only observed a single set of points so it is unsurprising that a larger sample size increases the difficulty for this problem.

\section{Methods and Simulations} \label{s:sim}

\subsection{Alternating LAP/TLS Algorithm}
We have been analyzing the theoretical properties of the estimator given by \eqref{eq:tls-estimator}, but since the solution involves a combinatorial optimization over $n!$ possible permutations, $\hat{\Pi}$ in general cannot be computed effiently.
In this section, we propose a simple and efficient algorithm, alternating LAP/TLS algorithm (ALTA), to approximate $\hat{\Pi}$, and use it to empirically examine our main result. The main idea of the algorithm is to alternate between estimating $(\hat{X},\hat{R})$ using the total least-squares (TLS) method and estimating $\hat{\Pi}$ by solving a linear assignment problem (LAP).

To motivate this approach, we first introduce a related algorithm, Alternating LAP/OLS algorithm (ALOA) for the case when $X$ is observed without noise.
Given an estimate for $\hat{R}$, in ALOA one estimates $\Pi^*$ as $\argmin_{\Pi\in \mathcal{P}_n}\|Y - \Pi X\hat{R}\|$.
Note this is simply a linear assignment problem with cost matrix given by $C$ where $C_{ij} = d(Y_{i}, (X\hat{R})_{j})$.
Given an estimate for $\Pi$, $\hat{R}$ can estimated using OLS.
The ALOA algorithm consists of alternating between these two steps.

For ALTA, given $\hat{\Pi}$, we can find $(\hat{X},\hat{R})$ using the total least squares method introduced in section \ref{Section: 2.1}.
Note, the ALTA algorithm is initialized at some permutation matrix $\Pi$ so the first step is to compute $\hat{X},\hat{R}$.

To find $\hat{\Pi}$ given $(\hat{X},\hat{R})$ is less well posed.
Indeed, recall that $\hat{X}, \hat{R}$ are not in the original formulation for the estimate $\hat{\Pi}$ given by Eq.~\eqref{eq:tls-estimator}.
Hence, we will use $(\hat{X},\hat{R})$ heuristically to estimate $\Pi^*$.
We propose solving a linear assignment problem based on a cost derived from $\hat{X}$ and $\hat{R}$.
The ALTA algorithm then iterates between the TLS-step to estimate $(\hat{X},\hat{R})$ and the LAP-step to estimate $\hat{\Pi}$ until convergence.
o specify the entries of the cost matrices in the LAP-step, since we have $Y_2=\Pi^*XR + E_2$, given $(\hat{X},\hat{R})$, we could consider a cost matrix, call it $C1$, taking it's $(i,j)$th entry as 
$$C1_{ij}=d(Y_{2i},\hat{X}\hat{R}_{j}) =\|Y_{2i}-\hat{X}\hat{R}_{j}\|_F^2.$$

On the other hand, since model \eqref{eq:tls-model} is equivalent to
\begin{align*}
    Y_1 &= \Pi^{*T}X + E_1\\
    Y_2 &= XR +E_2,
\end{align*}
we can estimate $\hat{\Pi}$ alternatively by solving the linear assignment problem between $\hat{X}$ and $Y_1$ with a cost matrix $C2$ given by
$$C2_{ij}=d(\hat{X}_i,Y_{1j})=\|\hat{X}_i-Y_{1j}\|_F^2.$$

The cost matrix $C3$ we arrive at is a combination of $C1$ and $C2$, and has it's $(i,j)$th entry being
$$C3_{ij}=d(Y_{2i},\hat{X}\hat{R}_{j}) + d(\hat{X}_i,Y_{1j}).$$

Note, $C1$, $C2$, and correspondingly $C3$ use both $\hat{X}$ and $\hat{R}$ from the TLS step.
From a different perspective, we could consider using only the $\hat{R}$ and re-estimate $X$ when computing the cost matrix.
We denote this as cost matrix $C4$ with entries
$$C4_{ij}=\min_{x\in\mathbb{R}^d}\|Y_{2i}-\hat{R}^Tx\|_F^2+\|Y_{1j}-x\|_F^2,$$
which represents the best estimation error for $x$ given $\hat{R}$ and $\hat{\Pi}_{ij}=1$.

\subsection{Simulation Studies}
Consider model \eqref{eq:tls-model} and assume that $\Pi^*=I_n$. We set $n=300$, $p=2$, and the observations $Y_1 = X + E_1$ and $Y_2 = \Pi^*XR + E_2$ are generated as follows.

The entries of the noise variables $E_1$ and $E_2$ are sampled from i.i.d. Gaussian distribution with mean zero and standard deviation $\sigma = 0.2$. As to the latent design matrix $X$, in order to meet the condition number requirement, $\kappa(X) = 1$, we first generate $X^\dagger$ with $X^\dagger_i\stackrel{i.i.d.}{\sim}N(0,I_p)$, letting $X^\dagger=USV^T$, we then take $X$ to be $X=\sqrt{np}\frac{U}{\|U\|_F}$.
This way, we have ensured that $\kappa(X)=\kappa(U)=1$ as well as $\|X\|_F=E(\|X^\dagger\|_F)=\sqrt{np}$. 
We have taken $R$ to be a 60 degree rotation matrix, and we initiate the algorithm by setting both the permutation matrix and the coefficient matrix to be the identity.

For the first 3 figures below, each algorithm is initialized at $\Pi^* = I_n$.
While this is unrealistic in practice, initializing at the truth is useful for investigating the thoretical properties of our estimator since this initialization increases the likelihood that we find an estimator that minimizes the sum of the least $p$ squared singular values.

In the following figures, unless otherwise stated, the vertical axis represents the estimation error measured by the normalized Procrustes quadratic loss, and the sample mean is averaged over 10 replications. The lines are produced by connecting the sample means and the confidence bars indicate the $(25\%, 75\%)$ confidence interval.

\begin{figure}
\centering
 \includegraphics[width=.666\linewidth]{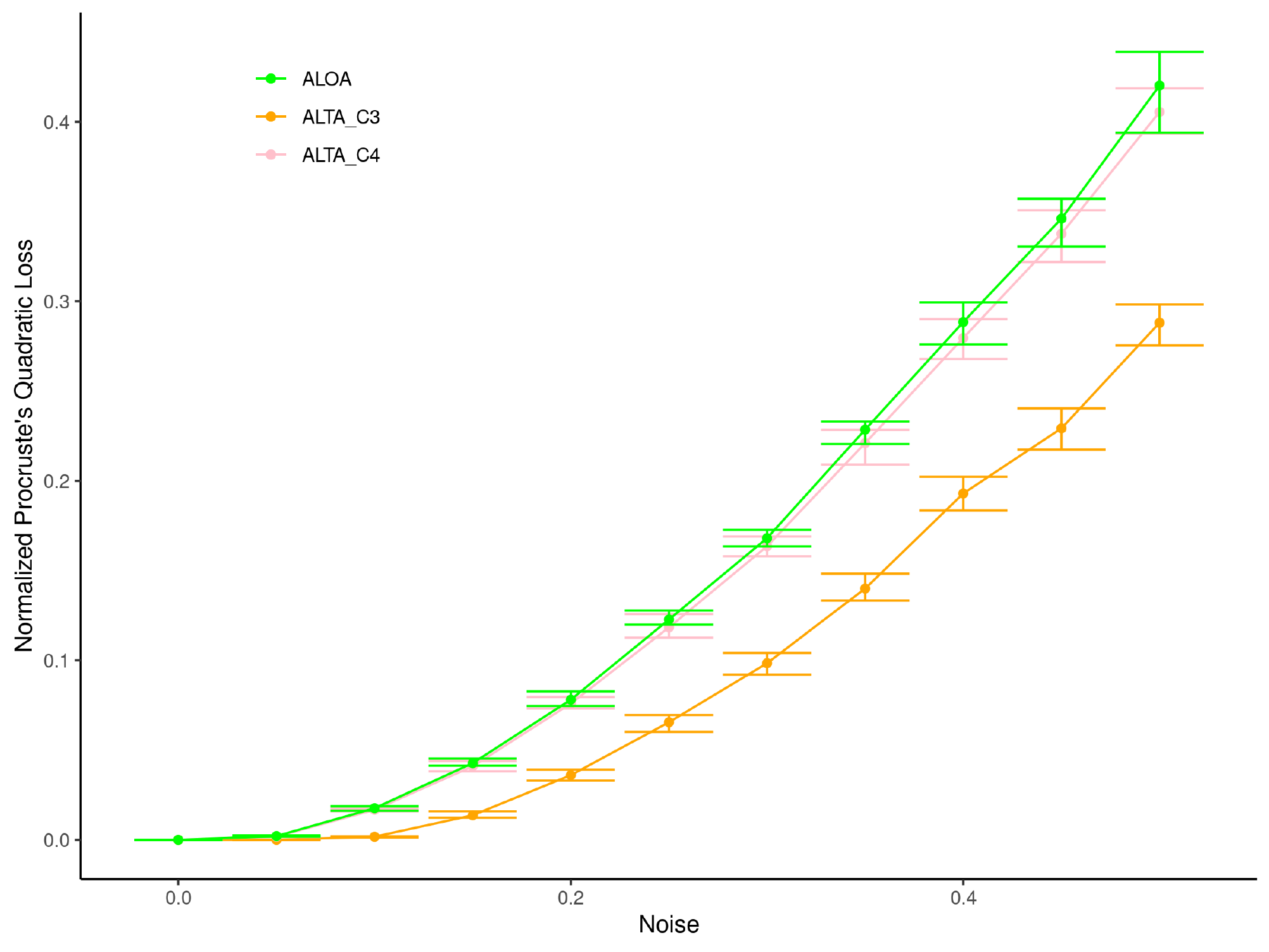}
  \caption{Performance comparison between different cost matrices for the ALTA-based method, as noise increases.}
  \label{fig:noise}
\end{figure}

As shown in Figure~\ref{fig:noise}, when we vary the noise level $\sigma$, the ALTA algorithm with cost matrix $C3$ has the best overall performance in terms of the normalized Procrustes quadratic loss. 
Note, when the true permutation matrix $\Pi^*$ is the identity, ALTA with cost matrix $C1$, $C2$, and $C3$, (in fact, also with ALOA), will have a good initial estimation for $(\hat{X},\hat{R})$, in comparison to when we initiate it at some random permutation matrix.

This is because in the TLS-step of the algorithm, which serves to estimate $(\hat{X},\hat{R})$, the initiation happens to be at the truth.
This is not so for ALTA\_C4 since it only utilizes $\hat{R}$ from the TLS-step, and re-estimate $(\hat{X},\hat{\Pi})$ in the LAP-step, resulting in a comparatively worse initial estimation.
We will examine the behavior of the methods when the true permutation matrix is not of the identity at the end of this section.
But for the purpose of evaluating the difference in performance between the ALTA algorithm and the ALOA algorithm, and to examine the upper abound in theorem 1, it is of reason and convenience to assume $\Pi^*=I$, and use ALTA\_C3 as the best algorithm to approximate $\Pi^*$.

Now to compare the performance between ALTA\_C3 and ALOA, we see that when there the noise is small, $\sigma<0.05$, both algorithms succeed in estimating the true permutation matrix in terms of the normalized Procruste's quadratic loss. However, as the noise increases, the underlying model that generates the observations deviates more and more from the OLS setting.
Therefore, as expected, the ALOA method performs increasingly worse in comparison with the ALTA\_C3 method. Hence, in real-world applications, ignoring the noise from one data source may result in poor performance in the recovery of the true permutation matrix.

\begin{figure}
\centering
\includegraphics[width=.666\linewidth]{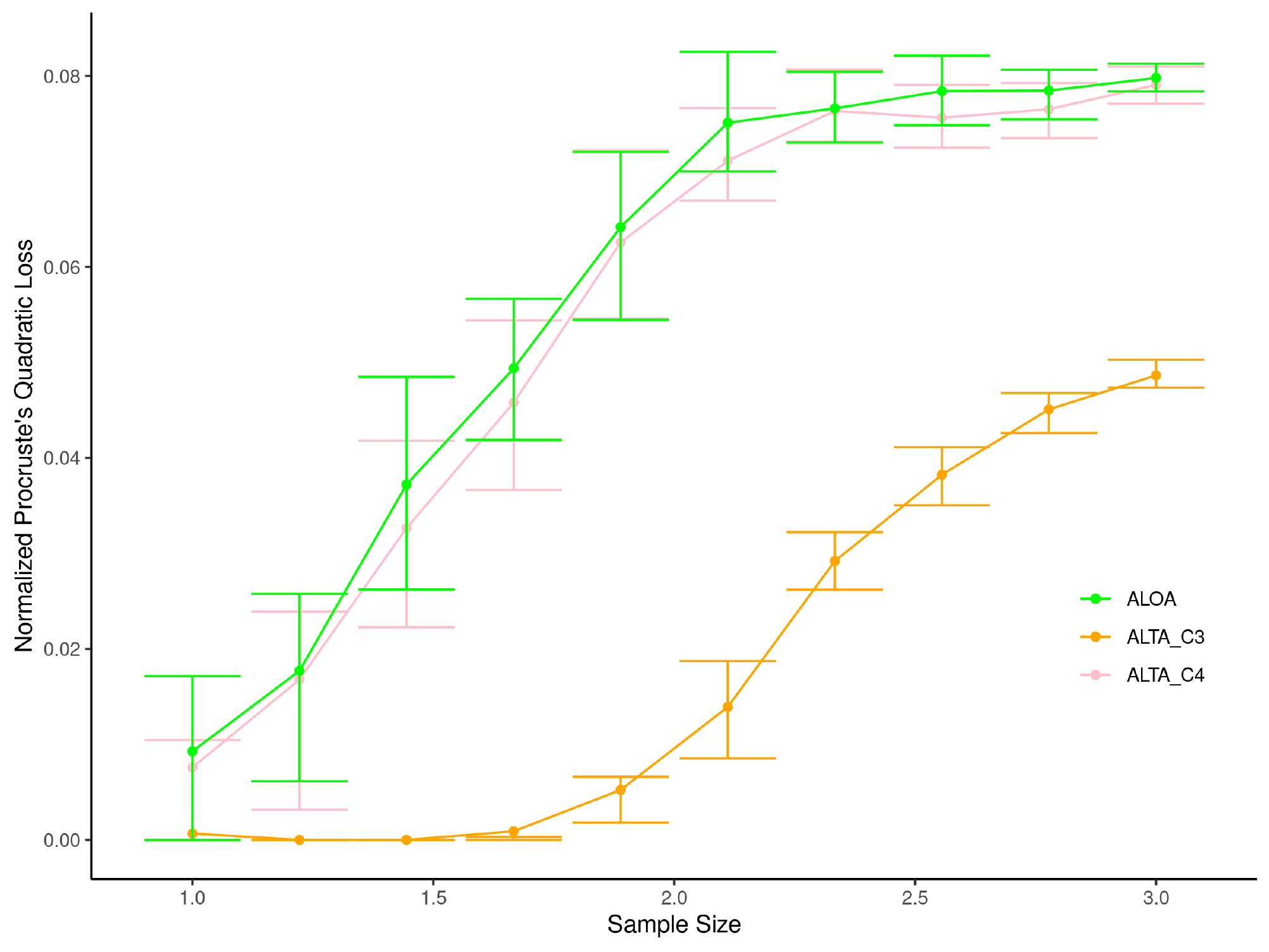}
  \caption{Performance comparison between ALTA and ALOA as sample size increases, with fixed noise.}
  \label{fig:sample_size}
\end{figure}
Figure \ref{fig:sample_size} compares the algorithms as the number of observations increases with a fixed noise level at $\sigma=0.2$. The horizontal axis represents 10 equally spaced sample size on the base-10 logarithmic scale. The sample mean is averaged over 30 Monte Carlo simulations.
Performance wise, it is as expected that ALTA\_C3 is still the best performing method across all sample sizes. ALTA\_C4 and ALOA resemble each other in performance likely due to the fact that in ALTA\_C4, the estimation of $\hat{X}$ is by nature utilizing the ordinary least-squares method.

\begin{figure}
\centering
\includegraphics[width=.666\linewidth]{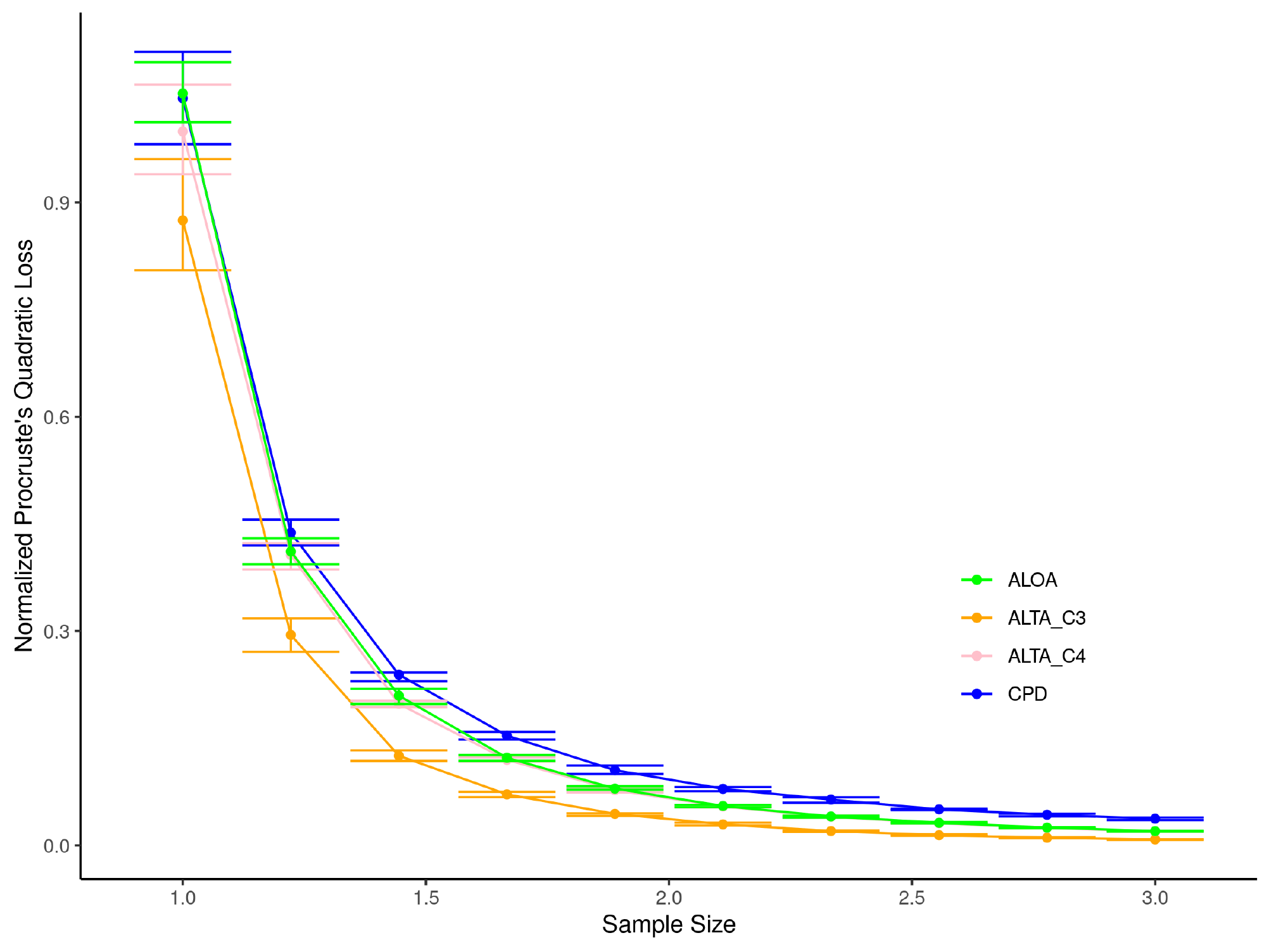}
  \caption{Performance comparison between ALTA, ALOA, and CPD as the signal-to-noise ratio increases.}
  \label{fig:snr}
\end{figure}
In figure \ref{fig:snr}, we compare the performance of the algorithms as we increase the signal-to-noise ratio via decreasing the noise level $\sigma$ in a $1/n$ fashion. We have compared the ALTA-based methods with both ALOA and a commonly used algorithm, and has demonstrated robust performances, known as coherent point drift (CPD), \citep{MS10}, which estimates $\Pi^*$ by treating one point set as the Gaussian center and the other as the data generated from them, and recover the correspondence using an EM-based procedure. As indicated by the theorem, we expect the normalized Procrustes quadratic loss to decrease and approaches zero as we increase the signal-to-noise ratio, which is just the behavior we have observed in the above figure. 

\begin{figure}
\centering
 \includegraphics[width=.666\linewidth]{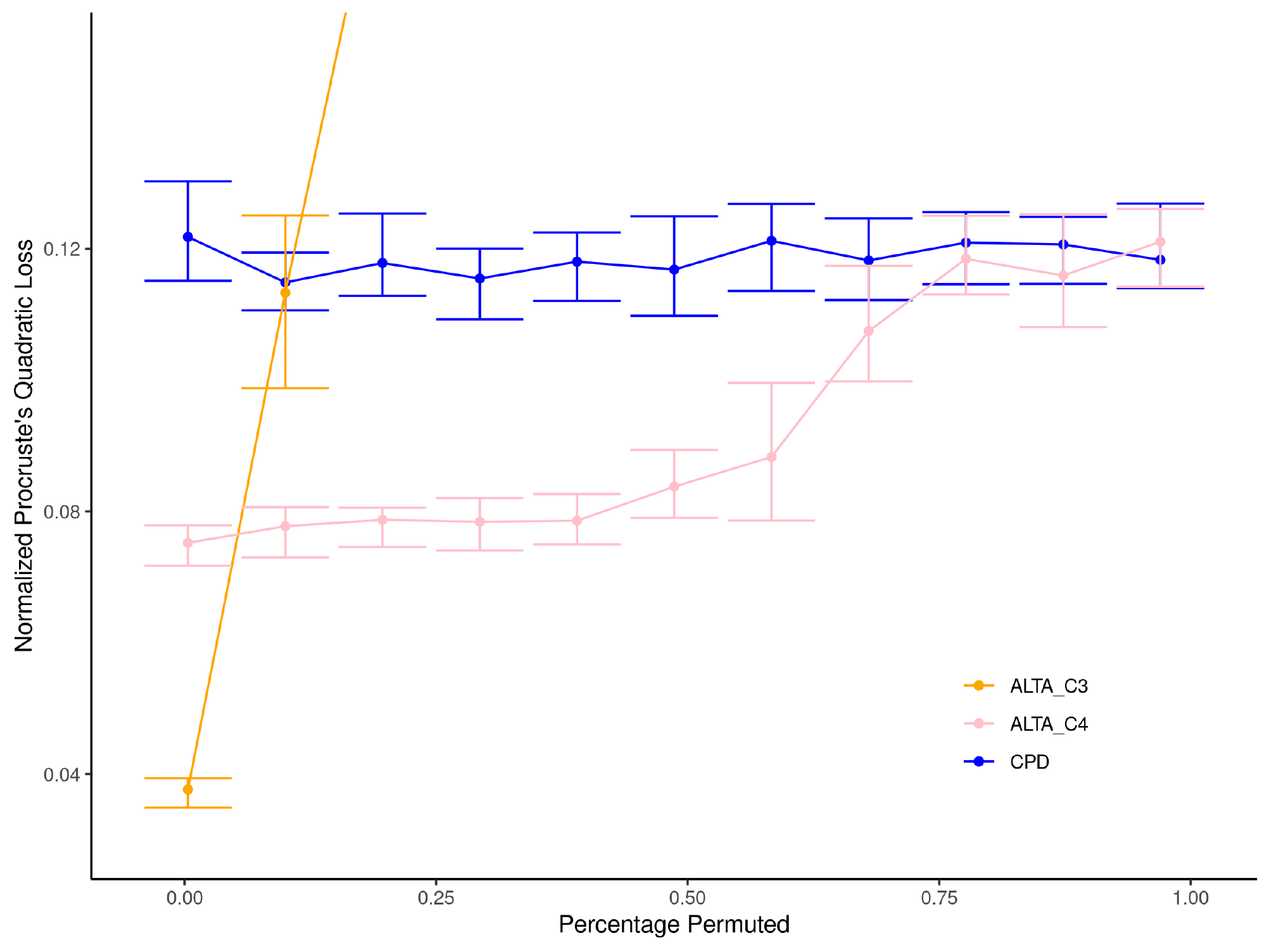}
  \caption{Performance comparison between ALTA-based methods versus CPD as the number of permutation increases.}
  \label{fig:permutation}
\end{figure}
In the last part of our simulation studies, we explore the behavior of the algorithms with more realistic initializations.
Figure~\ref{fig:permutation} compares the performance of ALTA as the relative Hamming distance between the initial permutation and the true permutation increases. We achieve this by setting the true permutation matrix as the identity matrix and increasing the number of top rows being randomly permuted.

When the percentage being permuted is very small, ALTA with cost matrix $C3$ performs better than ALTA\_C4, but such advantage quickly vanishes as the number of permutations increases.
This is because $C3$ ulitize the estimates $(\hat{X},\hat{R})$ from the TLS-step, as a consequence which they are more prone to be influenced by the quality of the initial permutation matrix.
On the other hand, ALTA with cost matrix $C4$ uses only $\hat{R}$ from the TLS-step and re-estimates
$(\hat{X},\hat{\Pi})$ in the LAP-step, therefore, it is comparatively more robust to the initiation. 

The CPD algorithm does not depend on the initiation and therefore remains robust the the percentage being permuted. However, ALTA\_C4 still outperforms CPD, especially when the percentage being permuted is less that $75\%$.

\section{Discussion}

We have analyzed the problem of permutation estimation in the total least squares setting and provided an upper bound on the normalized Procrustes quadratic loss for a total least squares estimator.
Moreover, we proposed a computationally efficient algorithm called LAP\_TLS to approximate the TLS estimator and used it to examine the upper bound empirically.
It is worth pointing out that the upper bound on the Procrustes quadratic loss that we observed in simulation are somewhat tighter than those provided by the theory.
The source of this gap might come from the following areas and which may represent valuable lines of inquiry in an effort to improve the bounds.

First, in the proof, we have employed a union bound over all the permutation matrices which resulted in us having $n^k$ permutation matrices to be unioned over.
However, we suspect that the number of permutation matrices could be much less than this since the set of $\Pi$ that satisfies equation \eqref{eq:5.3} should be a small set.
Secondly, to translate the inequality of \eqref{eq:5.5} to \eqref{eq:5.8}, we have used the Weyl's theorem which connects the sum of eigenvalues to the extreme eigenvalues.
Other bounds such as the Ky-Fan inequality, relate sums of eigenvalues to sums of eigenvalues and could tighten the bound.
In practice, this could improve the theoretical bounds in cases where the sum of eigenvalues is much less than $p$ times the largest eigenvalue.
This is likely in high-dimensional settings which are outside of the scope of this work.
The last area one might improve the bounds comes from inequality \eqref{eq:5.8} to \eqref{eq:5.9}, where we have dropped the term $\lambda_{2p}(E_{\hat{\Pi}}^TE_{\hat{\Pi}})$ due to its positivity.
If one considers the condition number of $E_{\hat{\Pi}}^TE_{\hat{\Pi}}$ which connects its minimum eigenvalue to its largest, the term $\lambda_{1}(E_{\Pi^*}^TE_{\Pi^*})-\lambda_{2p}(E_{\hat{\Pi}}^TE_{\hat{\Pi}})$ could also be made smaller.

With regard to the model assumptions, there are several extensions worth exploring.
We have assumed that the model is homoscedastic, but the case when the covariance of each observation depends on the latent design matrix, that is, the case of heteroscedasticity, is also of great interest.
Furthermore, considering the application of these results to the matching of two graphs, it is likely that the two graphs observed are correlated to some extent and that the variances of the estimates from the two graphs and across different vertices are not of the same.
That is, we might have $E_1$ and $E_2$ arising from non-identical heteroscedastic Gaussian distributions that are correlated.
Lastly, the observations, assumed to be independent, might in reality have some dependencies in them.

In the theorem, the property of the total least squares estimator is examined, but in practice, however, it is of no small difficulty to evaluate it computationally because it requires a combinatorial minimization over $n!$ permutations to find the TLS estimator.
So far no approaches can reliably approximate $\hat{\Pi}$, and therefore it is of great practical importance to design efficient algorithms to approximate $\hat{\Pi}$.
It will also be useful to better study the theoretical properties of the ALTA and ALOA approaches in terms of what can be said about the solutions they do attain.

Finally, we have not explored whether our results achieve or nearly achieve the minimax rate for this problem.
While this is the case, since our bounds compare very favorably to the bounds in the easier case of shuffled OLS, it is reasonable to conclude that our results are likely near the minimax rate for the shuffled TLS problem.

\section*{Acknowledgments}
This material is based on research sponsored by the Air Force Research Laboratory and Defense Advanced Research Projects Agency (DARPA) under agreement number FA8750-20-2-1001. The U.S. Government is authorized to reproduce and distribute reprints for Governmental purposes notwithstanding any copyright notation thereon. The views and conclusions contained herein are those of the authors and should not be interpreted as necessarily representing the official policies or endorsements, either expressed or implied, of the Air Force Research Laboratory and DARPA or the U.S. Government.

\bibliographystyle{plainnat}
\bibliography{references.bib}

\appendix

\section{Proof of Theorem~\ref{Theorem: 1}}
\label{Section: 5}

We begin by giving a high-level overview of the proof. The proof consists of two aspects. From the definition of the total least squares estimator, we know that $\sum_{i=p+1}^{2p}\sigma_i^2(Y_{\hat{\Pi}}) \le\sum_{i=p+1}^{2p}\sigma_i^2(Y_{\Pi^*})$ since $\hat{\Pi}$ is the minimizer among all possible permutation matrices. Using Weyl's theorem which provides both lower and upper bounds on the sum of eigenvalues of Hermitian matrices, we can derive an upper bound on the sum of the least $p$ eigenvalue of $M_{\hat{\Pi}}^TM_{\hat{\Pi}}$, where $M_{\hat{\Pi}}=[\Pi^*XR|\hat{\Pi}X]$, in terms of the Frobenius norm of $X$ and the largest eigenvalue of the error matrix, $\lambda_1(E_{\Pi^*}^TE_{\Pi^*})$, a quantity can be bounded using results from \citep{GT14} on tail bounds on the extreme eigenvalues. The second aspect of the proof upper bound our quantity of interest, that is, the Procrustes quadratic loss by the sum of the least $p$ eigenvalues of $M_{\hat{\Pi}}^TM_{\hat{\Pi}}$, and therefore passing the upper bound on the latter quantity to the former. 

Specifically, for $\Pi \in \mathcal{P}$, let
\begin{align}
\label{eq:5.2}
    Y_{ \Pi } = M_{ \Pi } + E_{ \Pi },
\end{align}
where $Y_{ \Pi } = [Y_2| \Pi  Y_1],M_{ \Pi }=[\Pi^*XR| \Pi  X],E_{ \Pi }=[E_2| \Pi  E_1]$.
By the definition of the shuffled total least-squares estimator \eqref{eq:tls-estimator}, $\hat{\Pi}$ minimizes the sum of the least $p$ squared singular values of $Y_\Pi$ among all possible $\Pi\in \mathcal{P}$, which implies that
\begin{align}
\label{eq:5.3}
    &\sum_{i=p+1}^{2p}\sigma_i^2(Y_{\hat{\Pi}}) \le\sum_{i=p+1}^{2p}\sigma_i^2(Y_{\Pi^*}),
\end{align}
which is equivalent to
\begin{align}
\label{eq:5.4}
\sum_{i=p+1}^{2p}\sigma_i^2(M_{\hat{\Pi}}+E_{\hat{\Pi}}) \le\sum_{i=p+1}^{2p}\sigma_i^2(M_{\Pi^*}+E_{\Pi^*}).
\end{align}
Since for any tall matrix $A$, we have $\sigma_i^2(A)=\lambda_i(A^TA)$, \eqref{eq:5.4} is equivalent to
\begin{equation}
\label{eq:5.5}
\begin{gathered}
    \sum_{i=p+1}^{2p}\lambda_i(M_{\hat{\Pi}}^T M_{\hat{\Pi}}+M_{\hat{\Pi}}^T E_{\hat{\Pi}}+E_{\hat{\Pi}}^T M_{\hat{\Pi}}+ E_{\hat{\Pi}}^T E_{\hat{\Pi}}) \\
    \quad\le\sum_{i=p+1}^{2p}\lambda_i(M_{\Pi^*}^T M_{\Pi^*}+M_{\Pi^*}^T E_{\Pi^*}+E_{\Pi^*}^T M_{\Pi^*}+E_{\Pi^*}^T E_{\Pi^*}).
\end{gathered}
\end{equation}
According to Weyl's theorem, \citep{HJ87}, for Hermitian matrices $A$ and $B$, it holds that $\lambda_k(A)+\lambda_{min}(B)\le \lambda_k(A+B)\le \lambda_k(A)+\lambda_{max}(B)$.
We can use this relationship to lower bound the left hand side of the inequality \eqref{eq:5.5} and upper bound its right hand side. 
Applying the bound to the left hand side gives us
\begin{align}
\label{eq:5.6}
    &\quad \sum_{i=p+1}^{2p}\lambda_i(M_{\hat{\Pi}}^TM_{\hat{\Pi}}+M_{\hat{\Pi}}^TE_{\hat{\Pi}}+E_{\hat{\Pi}}^TM_{\hat{\Pi}}+E_{\hat{\Pi}}^TE_{\hat{\Pi}}) \\ \notag
    &\ge \sum_{i=p+1}^{2p}\lambda_i(M_{\hat{\Pi}}^TM_{\hat{\Pi}})+p\lambda_{2p}(M_{\hat{\Pi}}^TE_{\hat{\Pi}}+E_{\hat{\Pi}}^TM_{\hat{\Pi}})+p\lambda_{2p}(E_{\hat{\Pi}}^TE_{\hat{\Pi}}),
\end{align}
and applying it to the right hand side gives us
\begin{align}
\label{eq:5.7}
    &\quad \sum_{i=p+1}^{2p}\lambda_i(M_{\Pi^*}^TM_{\Pi^*}+M_{\Pi^*}^TE_{\Pi^*}+E_{\Pi^*}^TM_{\Pi^*}+E_{\Pi^*}^TE_{\Pi^*})\\ \notag
    &\le \sum_{i=p+1}^{2p}\lambda_i(M_{\Pi^*}^TM_{\Pi^*})+p\lambda_{1}(M_{\Pi^*}^TE_{\Pi^*}+E_{\Pi^*}^TM_{\Pi^*})+p\lambda_{1}(E_{\Pi^*}^TE_{\Pi^*}).
\end{align}
Note also that $\sum_{i=p+1}^{2p}\lambda_i(M_{\Pi^*}^TM_{\Pi^*})=0$ since  $\mathrm{rank}(M_{\Pi^*})=p.$

All together these result in us having an upper bound on the sum of the least $p$ singular values of $M_{\hat{\Pi}}^TM_{\hat{\Pi}}$ as
\begin{align}
\label{eq:5.8}
    &\quad \frac{1}{p}\sum_{i=p+1}^{2p}\lambda_i(M_{\hat{\Pi}}^TM_{\hat{\Pi}})\notag\\
    &\le \lambda_{1}(M_{\Pi^*}^TE_{\Pi^*}+E_{\Pi^*}^TM_{\Pi^*})-\lambda_{2p}(M_{\hat{\Pi}}^TE_{\hat{\Pi}}+E_{\hat{\Pi}}^TM_{\hat{\Pi}})+\lambda_{1}(E_{\Pi^*}^TE_{\Pi^*})-\lambda_{2p}(E_{\hat{\Pi}}^TE_{\hat{\Pi}}).
\end{align}
Since $M_{\hat{\Pi}}^TE_{\hat{\Pi}}+E_{\hat{\Pi}}^TM_{\hat{\Pi}}$ is not necessarily positive semidefinite, we further bound this term using the relationship  $$\lambda_{2p}(M_{\hat{\Pi}}^TE_{\hat{\Pi}}+E_{\hat{\Pi}}^TM_{\hat{\Pi}})\ge -\sigma_{1}(M_{\hat{\Pi}}^TE_{\hat{\Pi}}+E_{\hat{\Pi}}^TM_{\hat{\Pi}}).$$
Moreover, since $\lambda_{1}(M_{\Pi^*}^TE_{\Pi^*}+E_{\Pi^*}^TM_{\Pi^*})\le \sigma_{1}(M_{\Pi^*}^TE_{\Pi^*}+E_{\Pi^*}^TM_{\Pi^*})$ and $\lambda_{2p}(E_\Pi^TE_\Pi)\ge 0$, therefore \eqref{eq:5.8} implies that

\begin{align}
\label{eq:5.9}
&\quad \frac{1}{p}\sum_{i=p+1}^{2p}\lambda_i(M_{\hat{\Pi}}^TM_{\hat{\Pi}})\notag\\
&\le \sigma_{1}(M_{\Pi^*}^TE_{\Pi^*}+E_{\Pi^*}^TM_{\Pi^*})+\sigma_{1}(M_{\hat{\Pi}}^TE_{\hat{\Pi}}+E_{\hat{\Pi}}^TM_{\hat{\Pi}})+\lambda_{1}(E_{\Pi^*}^TE_{\Pi^*}).
\end{align}
To provide upper bounds on $\sigma_{1}(M_{\Pi^*}^TE_{\Pi^*}+E_{\Pi^*}^TM_{\Pi^*})$ and $\sigma_{1}(M_{\hat{\Pi}}^TE_{\hat{\Pi}}+E_{\hat{\Pi}}^TM_{\hat{\Pi}})$, we use the relationship that for any matrix $A$ and $B$, we have, \citep{HJ91}, $\sigma_{1}(A+B)\le\sigma_{1}(A)+\sigma_{1}(B)$ and $\sigma_{1}(AB)\le\sigma_{1}(A)\sigma_{1}(B)$. Apply it to $\sigma_{1}(M_\Pi^T E_\Pi+E_\Pi^TM_\Pi)$, where $\Pi$ indictaes either $\hat{\Pi}$ or $\Pi^*$, we get 
\begin{align*}
    &\quad \sigma_{1}(M^{\Pi^T} E^\Pi+E^{\Pi^T}M^\Pi)\\
    &\le 2\sigma_{1}(M^{\Pi^T} E^\Pi)\\
    &= 2\sigma_{1}\begin{bmatrix}
    R^TX^T\Pi^{*^T}E_2&R^TX^T\Pi^{*^T}\Pi E_1\\
    X^T\Pi^TE_2&X^TE_1
    \end{bmatrix}\\
    &=2\sigma_{1}\begin{bmatrix}
    R^T&0\\
    0 & I_p\end{bmatrix}\begin{bmatrix}
    X^T\Pi^{*^T}E_2&X^T\Pi^{*^T}\Pi E_1\\
    X^T\Pi^TE_2&X^TE_1
    \end{bmatrix}\\
    &\le 2\max\{1,\sigma_{1}(R)\}\sigma_{1}(\begin{bmatrix}
    X^T\Pi^{*^T}E_2&X^T\Pi^{*^T}\Pi E_1\\
    X^T\Pi^TE_2&X^TE_1
    \end{bmatrix})\\
    &\le 2\max\{1,\sigma_{1}(R)\}\|\begin{bmatrix}
    X^T\Pi^{*^T}E_2&X^T\Pi^{*^T}\Pi E_1\\
    X^T\Pi^TE_2&X^TE_1
    \end{bmatrix}\|_F\\
    &\le 2\max\{1,\sigma_{1}(R)\}(\|X^T\Pi^{*^T}E_2\|_F+\|X^T\Pi^{*^T}\Pi E_1\|_F+\|X^T\Pi^TE_2\|_F+\|X^TE_1\|_F).
\end{align*}
Let$$\Gamma_{\hat{\Pi}}:=\|X^T\Pi^{*^T}E_2\|_F+\|X^T\Pi^{*^T}\hat{\Pi} E_1\|_F+\|X^T\hat{\Pi}^TE_2\|_F+\|X^TE_1\|_F,$$

and
$$\Gamma_{\Pi^*}:=\|X^T\Pi^{*^T}E_2\|_F+\|X^T E_1\|_F+\|X^T\Pi^{*^T}E_2\|_F+\|X^TE_1\|_F,$$
we therefore have,
\begin{align}
\label{eq:5.10}
    \sigma_{1}(M^{\hat{\Pi}^T} E^{\hat{\Pi}}+E^{\hat{\Pi}^T}M^{\hat{\Pi}})&\le 2\max\{1,\sigma_{1}(R)\}\Gamma_{\hat{\Pi}}\notag\\
    \sigma_{1}(M^{\Pi^{*^T}} E^{\Pi^*}+E^{\Pi^{*^T}}M^{\Pi^*})&\le 2\max\{1,\sigma_{1}(R)\}\Gamma_{\Pi^*}.
\end{align}
Plug \eqref{eq:5.10} into \eqref{eq:5.9}, and since from assumption \ref{Assumption:2} we have  $\max\{1,\sigma_{1}(R)\}=\sigma_{1}(R)$, \eqref{eq:5.9} then implies
\begin{align}
\label{eq:5.11}
    \frac{1}{p}\sum_{i=p+1}^{2p}\lambda_i(M_{\hat{\Pi}}^TM_{\hat{\Pi}})\le 2\sigma_1(R)(\Gamma_{\Pi^*}+\Gamma_{\hat{\Pi}})+\lambda_{1}(E_{\Pi^*}^TE_{\Pi^*}).
\end{align}
Next, we upper bound the Procrustes quadratic loss, $\min_{Q\in \mathcal{O}(p)} \| \Pi^*X-\hat{\Pi} X Q\|_F^2$, using the sum of the $p$ smallest singular values of $M_{\hat{\Pi}}^TM_{\hat{\Pi}}$. Consider the left hand side of\eqref{eq:5.11}, since
\begin{align}
\label{eq:5.12}
    \sum_{i=p+1}^{2p}\lambda_i(M_{\hat{\Pi}}^TM_{\hat{\Pi}}) &= \sum_{i=p+1}^{2p}\sigma_i^2(M_{\hat{\Pi}})\\ \notag
    &=\sum_{i=p+1}^{2p}\sigma_i^2([\Pi^*XR|\hat{\Pi}X])\\\notag
    &=\sum_{i=p+1}^{2p}\sigma_i^2([\Pi^*X|\hat{\Pi}X]\begin{bmatrix}
    R&0\\
    0&I
    \end{bmatrix})\\\notag
    &\ge \sum_{i=p+1}^{2p}\sigma_i^2([\Pi^*X|\hat{\Pi}X])(\min\{\sigma_p(R),1\})^2
\end{align}
Plug \eqref{eq:5.12} into \eqref{eq:5.12}, and since $\min\{\sigma_p(R),1\}=\sigma_p(R)$, we have a lower bound on the sum of the $p$ smallest singular values of $[\Pi^*X|\hat{\Pi}X]$:
\begin{align}
\label{eq:5.13}
    \frac{1}{p}\sum_{i=p+1}^{2p}\sigma_i^2([\Pi^*X|\hat{\Pi}X])\le \frac{1}{\sigma_p^2(R)}\left[2\sigma_1(R)(\Gamma_{\Pi^*}+\Gamma_{\hat{\Pi}})+\lambda_{1}(E_{\Pi^*}^TE_{\Pi^*})\right].
\end{align}
Now Lemma~\ref{lem:pql_svals} tells us that when $\kappa(X)=1$, it holds that
$$\min_{Q\in \mathcal{O}(p)} \| \Pi^*X-\Pi X Q\|_F^2 \le2\sum_{i=1+p}^{2p} \sigma_i^2(\Pi^*X|\Pi X).$$
Then \eqref{eq:5.13} indicates
\begin{align}
\label{eq:5.14}
    &\frac{\min_{Q\in \mathcal{O}(p)} \| \Pi^*X-\hat{\Pi} X Q\|_F^2}{\|X\|_F^2}\le \frac{2p}{\sigma_p^2(R)\|X\|_F^2}\left[2\sigma_1(R)(\Gamma_{\hat{\Pi}}+\Gamma_{\Pi^*})+\lambda_1(E_{\Pi^*}^TE_{\Pi^*})\right].
\end{align}
For the last part of the proof, we provide probabilistic upper bounds on the three random terms $\Gamma_{\hat{\Pi}},\Gamma_{\Pi^*}$,and $\lambda_{1}(E_{\Pi^*}^TE_{\Pi^*})$ on the right hand side of \eqref{eq:5.14} using results from the lemmas.

To upper bound the error term $\lambda_1(E_{\Pi^*}^TE_{\Pi^*})$, we use Lemma~\ref{lem:evals} which says that 
\begin{align*}
    \Pr\left(\lambda_1(E_{\Pi^*}^TE_{\Pi^*})\ge 2n\lambda_1(\Sigma)(1+\epsilon)\right) \le 2p*\exp\left(\frac{-cn\epsilon^2\lambda_1(\Sigma)}{\|\Sigma\|_*}\right)~for~\epsilon \le 4n,
\end{align*}
where $c\ge 1/32$.\\
Take $\epsilon=\sqrt{\frac{\eta\log(n)\tr(\Sigma)}{cn\lambda_1(\Sigma)}}$, it gives,
\begin{align}
\label{eq:5.15}
    \Pr\left(\lambda_1(E_{\Pi^*}^TE_{\Pi^*})\ge 2n\lambda_1(\Sigma)(1+a_n \eta)\right) \le \frac{2p}{n^{\eta^2}}.
\end{align}
where $a_n = \sqrt{\frac{\tr(\Sigma)}{\lambda_1(\Sigma)}\frac{\log(n)}{c n}}$

Next we bound $\Gamma_{\hat{\Pi}}$ and $\Gamma_{\Pi^*}$.
Each term here can be bounded using
\begin{align*}
    &\quad\|X^T\hat{\Pi}E_2\|_F^2\\
    &\leq \|X^T \hat{\Pi}\|_F^2 \lambda_1(E_2^TE_2)\\
    &\leq \|X\|_F^2\left(2n\lambda(\Sigma)\left(1+\eta a_n\right)\right),
\end{align*}
where the second inequality is with high probability by Eq.~\eqref{eq:5.15}.
Hence, provided the event in Eq~\eqref{eq:5.15} does not occur, it holds that
\begin{align*}
   &\quad \Gamma_{\Pi^*} + \Gamma_{\hat{\Pi}}\\
   &\leq 8\|X\|_F\sqrt{2n\lambda(\Sigma)\left(1+\eta a_n\right)}.
\end{align*}

Plugging \eqref{eq:5.15}  and the above into \eqref{eq:5.14} gives an upper bound of the normalized Procrustes quadratic loss as 

\begin{align}
    &\frac{\min_{Q\in \mathcal{O}(p)} \| \Pi^*X-\hat{\Pi} X Q\|_F^2}{\|X\|_F^2}\notag\\
    &\le \frac{2p}{\sigma_p^2(R)\|X\|_F^2}\left(2\sigma_1(R)(\Gamma_{\hat{\Pi}}+\Gamma_{\Pi^*})+\lambda_1(E_{\Pi^*}^TE_{\Pi^*})\right)\notag\\
    &\le \frac{2p}{\sigma_p^2(R)\|X\|_F^2}\left[16\sigma_1(R) \|X\|_F\sqrt{2n\lambda_1(\Sigma)\left(1+\eta a_n\right)} +2n\lambda_1(\Sigma)\left(1+\eta a_n\right)\right]
    \notag\\
    &\le \frac{2p}{\sigma_p^2(R)\|X\|_F^2}\left(1+\eta a_n\right)\lambda_1(\Sigma)\left[16\sigma_1(R) \|X\|_F\sqrt{2n} +2n\right]
\end{align}
with probability at least $1 - n^{-\eta^2}$.

\section{Technical Lemmas}\label{app:tech_lemma}

Here we state and prove the technical lemmas needed to prove theorem 1.

\begin{lemma}\label{lem:evals}
Assume $E^\Pi=[E_2|\Pi E_1], E_1\sim E_2\sim N(0,\Sigma)$, let $\Delta^{\Pi} =  E^{\Pi^T}E^\Pi$, then,
$$\Pr\left(\lambda_1(\Delta^{\Pi^*})\ge 2n\lambda_1(\Sigma)(1+\epsilon)\right) \le 2p*exp\left(\frac{-cn\epsilon^2\lambda_1(\Sigma)}{\|\Sigma\|_*}\right)~for~\epsilon \le 4n.$$
where $c$ is at least $1/32$.
\end{lemma}
\begin{proof}
Let
$$\Delta^\Pi =  E^{\Pi^T}E^\Pi=\begin{bmatrix}
E_2^T E_2 & E_2^T\Pi E_1 \\
E_1^T\Pi ^T E_2 &E_1^T E_1
\end{bmatrix},$$
we have the relationship:
$$(\frac{\sigma_1(E^\Pi)}{n})^2=\frac{\lambda_1(\Delta^\Pi)}{n^2}.$$
Now, from (Topics in Matrix Analysis, Problem 3.5.22), we know that $$\lambda_1(\Delta^{\Pi})\le \lambda_1(E_2^T E_2)+\lambda_1(E_1^T E_1);$$ 
From (Thm 7.1, Gittens \& Tropp, 2011), we have
$$P\left(\frac{\lambda_1(E_1^T E_1^T)}{n}\ge (1+\epsilon)\lambda_1(\Sigma)\right) \le p*exp\left(\frac{-cn\epsilon^2}{\Sigma_{i=1}^p\frac{\lambda_i(\Sigma)}{\lambda_1(\Sigma)}}\right)~for~\epsilon \le4n,$$
Therefore,
\begin{align*}
    &P\left(\frac{\lambda_1(\Delta^\Pi)}{n} \ge 2(1+\epsilon)\lambda_1(\Sigma)\right)
    \\&\leq P\left(\frac{\lambda_1(E_2^T E_2)+\lambda_1(E_1^T E_1^T)}{n}\ge (1+\epsilon)\lambda_1(\Sigma)+(1+\epsilon)\lambda_1(\Sigma)\right)\\
    &\le P\left(\frac{\lambda_1(E_2^T E_2^T)}{n}\ge (1+\epsilon)\lambda_1(\Sigma)\bigcup  \frac{\lambda_1(E_1^T E_1)}{n}\ge (1+\epsilon)\lambda_1(\Sigma)\right)\\
    &= P\left(\frac{\lambda_1(E_2^T E_2^T)}{n}\ge (1+\epsilon)\lambda_1(\Sigma)\right)+P\left( \frac{\lambda_1(E_1^T E_1)}{n}\ge (1+\epsilon)\lambda_1(\Sigma)\right)\\
    &\le 2p*exp\left(\frac{-2cn\epsilon^2}{\Sigma_{i=1}^p\frac{\lambda_i(\Sigma)}{\lambda_1(\Sigma)}}\right)\\&
    =2p*exp\left(\frac{-2cn\epsilon^2\lambda_1(\Sigma)}{\|\Sigma\|_*}\right),~for~\epsilon \le4n.
\end{align*}
\end{proof}

\begin{lemma} \label{lem:pql_svals}
Assume $\kappa(X)=1$, we have the relationship
$$\min_{Q\in \mathcal{O}(p)} \| X-\Pi X Q\|_F^2 \le2\sum_{i=1+p}^{2p} \sigma_i^2(X|\Pi X).$$
\end{lemma}
\begin{proof}
\begin{align*}
\sum_{i=1+p}^{2p} \sigma_i^2(X|\Pi X)
&= \min_{U,V\in \Re^{p\times p}: U^TU+V^TV = I} \left\|(X|\Pi X) \begin{pmatrix} U \\ V \end{pmatrix}\right\|_F^2 \\
&= \min_{U,V\in \Re^{p\times p}: U^TU+V^TV = I} \left\|(X|-\Pi X) \begin{pmatrix} U \\ V \end{pmatrix}\right\|_F^2 \\
&= \min_{U,V\in \Re^{p\times p}: U^TU+V^TV = I} \tr(U^TX^TX U + V X^TXV - 2U^TX^T \Pi X V) \\
&> \min_{U,V\in \Re^{p\times p}: U^TU+V^TV = I} \tr(U^TX^TX U + V X^TXV)\\
&\quad- \max_{U,V\in \Re^{p\times p}: U^TU+V^TV = I} 2\tr(U^TX^T \Pi X V) \\
&> \min_{U,V\in \Re^{p\times p}: \|U\|_F^2 +\|V\|^2 = p} \tr(U^TX^TX U + V X^TXV)\\
&\quad- \max_{U,V\in \Re^{p\times p}: U^TU+V^TV = I} 2\tr(U^TX^T \Pi X V) \\
&= p\sigma_{min}(X)^2 - \max_{U,V\in \Re^{p\times p}: U^TU+V^TV = I} 2\tr(U^TX^T \Pi X V)\\
&> \|X\|_F^2/\kappa(X)^2 - \max_{U,V\in \Re^{p\times p}: U^TU+V^TV = I} 2\tr(U^TX^T \Pi X V)\\
& =^{(i)} \|X\|_F^2 - \max_{U,V\in \Re^{p\times p}: U^TU+V^TV = I} 2\tr(U^TX^T \Pi X V)\\
& =^{(ii)}\|X\|_F^2 - \max_{Q\in \mathcal{O}(p)} \tr(X^T\Pi X Q)\\
&=^{(iii)}\frac{1}{2}\min_{Q\in \mathcal{O}(p)} \| X-\Pi X Q\|_F^2
\end{align*}
Step (i) follows since $\kappa(X)=1$.
Step (ii) is due to Lemma~\ref{lem:max_uv}.

Step (iii) follows since,
\begin{align*}
&\min_{Q\in \mathcal{O}(p)} \| X-\Pi X Q\|_F^2 
\\= & \min_{Q\in \mathcal{O}(p)}2(\|X\|_F^2-\tr(X^T\Pi X Q)).
\end{align*}
Hence we reach at,
$$\min_{Q\in \mathcal{O}(p)} \| X-\Pi X Q\|_F^2 \le2\sum_{i=1+p}^{2p} \sigma_i^2(X|\Pi X).$$
\end{proof}
\begin{lemma}\label{lem:max_uv}
    Under the same assumptions and notation as Lemma~\ref{lem:pql_svals},
    $$\max_{U,V\in \Re^{p\times p}: U^TU+V^TV = I} 2\tr(U^TX^T \Pi X V)=\max_{Q\in \mathcal{O}(p)} \tr(X^T\Pi X Q).$$
\end{lemma} 
\begin{proof}
Writing the singular value decomposition of $X^T\Pi X $ as $U_\Pi\Sigma_\Pi V_\Pi^T$, we have,
\begin{align*}
     &\quad\max_{U,V\in \Re^{p\times p}: U^TU+V^TV = I} 2\tr(X^T \Pi X VU^T)\\&=\max_{U,V\in \Re^{p\times p}: U^TU+V^TV = I} 2\tr(U_\Pi \Sigma_\Pi V_\Pi^T VU^T)\\
     &=\max_{U,V\in \Re^{p\times p}: U^TU+V^TV = I} 2\tr(\Sigma_\Pi V_\Pi^T VU^T U_\Pi)\\
     &=\max_{U,V\in \Re^{p\times p}: U^TU+V^TV = I} 2 \sum_{i=1}^{p}(\sigma_i(\Sigma_\Pi)*(V_\Pi^T VU^T U_\Pi)_{ii})\\
     &=^{(i)}\max_{U,V\in \Re^{p\times p}: U^TU+V^TV = I} 2\sigma_1(\Sigma_\Pi)* \sum_{i=1}^{p}(V_\Pi^T VU^T U_\Pi)_{ii}\\
     &=\max_{U,V\in \Re^{p\times p}: U^TU+V^TV = I} 2\sigma_1(\Sigma_\Pi)* \tr(V_\Pi^T VU^T U_\Pi)
     \end{align*}
Step (i) follows since we have $\kappa(X)=\kappa(\Pi X)=1$, which implies that 
$$1 \le \kappa(X^T\Pi X)=\frac{\sigma_{max}(X^T\Pi X)}{\sigma_{min}(X^T\Pi X)}\le \frac{\sigma_{max}(X^T)\sigma_{max}(\Pi X)}{\sigma_{min}(X)\sigma_{min}(\Pi X)}\le 1$$
then, $\kappa(X^T\Pi X)=1,$
indicating that $\sigma_1(\Sigma_\Pi)=...=\sigma_p(\Sigma_\Pi)$.

Now let,
\begin{align*}
    g_1 = \tr(V_\Pi^T VU^T U_\Pi)\\
    g_2 = \tr[L(U^TU+V^TV-I)]
\end{align*}
where the $(p\times p)$ matrix $L$ is a matrix of (unknown) Lagrange multipliers.
The function $g$ to be differentiated partially with respect to the elements of $U$ and $V$ is then
\begin{align*}
    g = g_1 + g_2
\end{align*}
Recall trace derivative rules:
\begin{align*}
    \nabla_X\tr(AXB)=A^TB^T\\
    \nabla_X\tr(AX^TB)=BA\\
    \nabla_X\tr(BX^TX)=XB^T+XB
\end{align*}
Therefore,
\begin{align*}
    \nabla_V \tr(V_\Pi^T VU^T U_\Pi-L(U^TU+V^TV-I))|_{\Bar{U},\Bar{V}}=0\\
    \nabla_U \tr(V_\Pi^T VU^T U_\Pi-L(U^TU+V^TV-I))|_{\Bar{U},\Bar{V}}=0
\end{align*}
Gives,
\begin{align*}
    V_\Pi U_\Pi^T \Bar{U} + \Bar{V} (L+L^T) =0\\
    U_\Pi V_\Pi ^T\Bar{V}  + \Bar{U}(L+L^T) =0
\end{align*}
Set $Q = L+L^T, A = U_\Pi^T\Bar{U}, B = V_\Pi^T\Bar{V}$, we can write it as
\begin{align*}
        A + BQ = 0 -(1)\\
        B + AQ = 0 -(2)\\
        A^TA+B^TB=I -(3)
\end{align*}
First, we have
\begin{align*}
    (1) &\implies AQ+BQQ^T = 0\\
    +(2) &\implies QQ^T = I
\end{align*}
Also, note that $Q$ is also symmetric, hence $Q$ is a diagnocal matrix consists of 1's and -1's.
Also,
\begin{align*}
  -A^TBQ - Q^TA^T B = I\\
\implies A^TB = -\frac{Q}{2}
\end{align*}
Note that our constrained optimization problem now becomes,
\begin{align*}
    \max \tr(A^TB)=\max \tr(-\frac{Q}{2})\\
    s.t ~ A^TA+B^TB=I
\end{align*}
Therefore the maximum is taking place when $Q=-\frac{I}{2}$, and we have 
\begin{align*}
   & A^TA=B^TB=\frac{I}{2}\\
    &\implies \Bar{U}^T\Bar{U} = \Bar{V}^T\Bar{V} = \frac{I}{2}\\
    &\implies \Bar{U}=\frac{Q_1}{\sqrt{2}}, \Bar{V}=\frac{Q_2}{\sqrt{2}}, Q_1^T Q_1=Q_2^T Q_2=I
\end{align*}
Therefore
\begin{align*}
         \max_{U,V\in \Re^{p\times p}: U^TU+V^TV = I} 2\tr(V_\Pi^T VU^T U_\Pi) &= \max_{Q_1, Q_2 \in \mathcal{O}(p)}  2 \tr(V_\Pi^T \frac{Q_1Q_2^T}{2}U_\Pi)\\
         &= \max_{Q \in \mathcal{O}(p)}   \tr(V_\Pi^T QU_\Pi)
\end{align*}
Hence,
\begin{align*}
    \max_{U,V\in \Re^{p\times p}: U^TU+V^TV = I} 2\tr(X^T \Pi X VU^T)&=\max_{U,V\in \Re^{p\times p}: U^TU+V^TV = I} 2\tr(U_\Pi \Sigma_\Pi V_\Pi^T VU^T)\\
    &=\max_{U,V\in \Re^{p\times p}: U^TU+V^TV = I} 2\sigma_1(\Sigma_\Pi)* \tr(V_\Pi^T VU^T U_\Pi)\\
 &= \max_{Q \in \mathcal{O}(p)}  \sigma_1(\Sigma_\Pi)* \tr(V_\Pi^T QU_\Pi)\\
 &=\max_{Q\in \mathcal{O}(p)} \tr(X^T\Pi X Q)
\end{align*}
\end{proof}

\end{document}